\def \version {2021--05--07}

\documentclass[12pt]{article}
\usepackage{epic,latexsym}
\usepackage{amssymb,amsmath,amsthm}
\usepackage{color}
\usepackage{tikz}

\newcommand{\eps}{\varepsilon}

\newtheorem{thm}{Theorem}
\newtheorem{lem}[thm]{Lemma}

\newtheorem{prop}[thm]{Proposition}
\newtheorem{rmk}[thm]{Remark}
\newtheorem{conj}{Conjecture}

\def \bsk {\bigskip}

\def \msk {\medskip}
\def \es {\emptyset}

\def \nmr {\begin{enumerate}}
\def \enmr {\end{enumerate}}

\def \tmz {\begin{itemize}}
\def \etmz {\end{itemize}}

\newcommand{\cG}{{\cal G}}
\newcommand{\cE}{{\cal E}}

\newcommand{\Proof}{\noindent\textbf{Proof. }}
\newcommand{\cH}{{\cal H}}

\newcommand{\ft}{\tau^*}

\begin{document}

\title{Optimal strategies in fractional games:\\
   vertex cover and domination
}
\author{Csilla Bujt\'as\,\thanks{Faculty of Mathematics and Physics,
 University of Ljubljana, Slovenia, and
University of Pannonia, Veszpr\'em, Hungary.
}\,\,, \ \
 G\"unter Rote\,\thanks{Institut f\"ur Informatik,
  Freie Universit\"at Berlin, Takustr.~9, 14195 Berlin, Germany.
  }\,\,, \ \
 Zsolt Tuza\,\thanks{Alfr\'ed R\'enyi Institute of Mathematics,
  1053 Budapest, Re\'altanoda u.~13--15, Hungary;
   and
  Department of Computer Science and Systems Technology,
   University of Pannonia, 8200 Veszpr\'em, Egyetem u.~10, Hungary.
   }}
\date{\small Latest update on \version}
\maketitle

\begin{abstract}

In a hypergraph $\cH=(V,\cE)$ with vertex set $V$ and edge set $\cE$,
 a real-valued function $f: V \to [0,1]$ is a fractional transversal if
 $\sum_{v \in E} f(v) \ge 1$ holds for every $E\in\cE$.
Its size is $|f|:=\sum_{v\in V} f(v)$, and the fractional transversal number
 $\ft(\cH)$ is the smallest possible $|f|$.
 
We consider a game scenario where two players with opposite goals
 construct a fractional transversal incrementally, trying to minimize and
  maximize $|f|$, respectively.
We prove that both players have strategies to achieve their common optimum,
 and they can reach their goals using rational weights.
 \vspace*{3ex}

\noindent {\small \textbf{Keywords:} fractional vertex cover,
  fractional transversal game,  fractional domination \vspace*{1ex} game.} 
\\
\noindent {\small \textbf{AMS subject classification:}  05C69, 05C65, 05C57}

 \end{abstract}

\newpage

\section{Introduction}  

Let $\cH=(V,\cE)$ be a finite hypergraph, where $V$ is the finite vertex set
 and  $\cE$ is the edge set, a set system over the underlying set $V$.
We assume that every edge contains at least one vertex; that is,
 $\cE \subseteq 2^{V}\setminus \{\es \}$.
A hypergraph is \emph{$k$-uniform} if $|E|=k$ holds for all $E\in\cE$.
A set $T\subseteq V$ is a \emph{transversal}\,\footnote{In various areas of
 discrete mathematics and computer science, the equivalent standard name of
  transversal is vertex cover, or hitting set, or blocking set.
It is also equivalent to the set cover in the dual hypergraph.}
  of $\cH$ if every edge is covered by a vertex of $T$, which formally
   means that $T\cap E\neq \es$ holds for all $E\in\cE$.
Its real relaxation, called \emph{fractional transversal}, is a
 function $f\colon V \to [0,1]$ such that
 $\sum_{v \in E} f(v) \ge 1$ holds for every $E\in\cE$.
 The size of $f$ is defined as $|f|:=\sum_{v\in V} f(v)$.
The \emph{transversal number}  $\tau(\cH)$ and the
\emph{fractional transversal number} $\ft(\cH)$ of $\cH$ are
 the minimum cardinality $|T|$ of a transversal and minimum value $|f|$ of a
fractional transversal,
respectively.

The \textit{transversal game}
is a competitive optimization version of hypergraph transversals, which
was introduced in~\cite{BHT-2016} and studied
further in~\cite{BHT-2017}.
It is played on a hypergraph $\cH$ by two players called Edge-hitter and Staller.
They take turns choosing a vertex.
The game is over when all edges are covered, and the length of the game
is the number of vertices chosen by the players.
Edge-hitter wants to finish the game as soon as possible, while
Staller wants to delay the end.
To prevent Staller from making completely useless moves, we stipulate
that the chosen vertex must be
contained in at least one
previously uncovered edge.

Assuming that both players play optimally and Edge-hitter starts, the
length of the game on $\cH$ is uniquely determined. It is called the
\textit{game transversal number} of $\cH$ and is denoted by $\tau_g(\cH)$.
Analogously, the \emph{Staller-start game transversal number} of $\cH$,
denoted by $\tau_g'(\cH)$, is the length of the game under the same rules
when Staller makes the first move. Among other results, it was proved in~\cite{BHT-2016} that $|\tau_g(\cH) - \tau_g'(\cH)| \le 1$ always holds. We further recall that, denoting by $n(\cH)$ and $m(\cH)$ the number of vertices and edges in $\cH$ respectively, $\frac{4}{11}(n(\cH)+m(\cH))$ is a (sharp) upper bound on $\tau_g(\cH)$ if $\cH$ does not contain one-element edges and it is not isomorphic to the cycle~$C_4$.

Below we shall refer to this game as the \textit{integer game},
as opposed to its fractional version which we will introduce
 in the next section.

Important motivation of this approach are the \emph{domination game}~\cite{BKR-2010} and the \emph{total domination game}~\cite{HKR-2015}, where in fact the transversal game is
 played on the `closed neighborhood hypergraph' and on the `open neighborhood
  hypergraph' of a graph, respectively.
Further variants studied so far include the
disjoint domination~\cite{BT-2016}, connected domination~\cite{borowiecki-2019}, and
 fractional domination~\cite{BT-fr} games on graphs, and the domination 
 games  on hypergraphs~\cite{BPTV}. Some of the most recent results  can be found in \cite{BBGK-2019, BDKK, bujtas-2018, BIK-2021, henning-2016, I-2019, kinnersley-2013, KR-2019, Kos-2017, XLK-2018}.  For a thorough survey and list of further references see the book \cite{BHKR-book}.

\subsection*{Our results}

In Section \ref{sec:fractransv} we introduce the rules of the game,
 prove that its value is well-defined, and present some examples.
The latter also show that it matters which player starts, moreover
 edges that are subsets of other edges of the hypergraph may influence the
 game value, which is not the case for the standard non-game version
 of transversal number.

In Section \ref{sec:contprinc} we compare the game transversal number
 with other related parameters, and prove a monotonicity property,
 implying that changing the starting player
can affect the value of the game by at most 1.

The rules of the game allow both players to split their moves into
 infinitely many submoves.
In Section \ref{sec:infinite} we prove that any infinite move is
 equivalent to a finite move.

In Section \ref{sec:algo} we prove that the game can be modeled in a way
 that leads to an optimization problem solvable via the theory of
 piecewise linear continuous rational functions.
From this we derive that the game value is rational for every finite
 hypergraph, moreover both players can achieve their goals using
 rational submoves.

Consequences on domination games and several conjectures are given in the
 concluding section.

\section{Fractional transversal game}
\label{sec:fractransv}

Let $\cH=(V,\cE)$ be a hypergraph. In the context of the fractional transversal game, we will consider a \emph{(partial) cover function}
  $t\colon V \to [0,1]$ that is updated after each move during the game.
  We denote by $|t|$ the sum $\sum_{v\in V}t(v)$.
Given a cover function~$t$, the corresponding \emph{load function} is $\ell\colon \cE \to [0,1]$ defined by the rule
  $$\ell(E)=\ell(E,t)=\min\{1, \sum_{v\in E} t(v)\}$$
  for every $E \in \cE$.
If $\ell\equiv 1$, we say that $\cH$ is fully covered. We shall write $t_i$ and $\ell_i$ for the cover and load functions obtained after the $i^{\rm th}$ move.

The game begins with $t_0\equiv 0$  and therefore with $\ell_0\equiv 0$.
 It is finished when the hypergraph becomes fully covered.
Edge-hitter and Staller take turns making moves under the following rules.
As long as $\ell\not\equiv 1$, the next player performs a
 \emph{move} which is a sequence $(v_{i_1},w_1),(v_{i_2},w_2), \dots $ of
arbitrary length (possibly infinite). It consists of the \emph{submoves} $(v_{i_k},w_k)$, $k=1,2,\dots,$ where $v_{i_1}, v_{i_2},\dots $ are vertices of $\cH$ with any number of repetitions allowed, and the \emph{weights}
 $w_1,w_2,\dots$ are real numbers from $[0,1]$.
  It is required that
$$\sum_{k\ge 1} w_k=1$$ in each move except the last one, whereas it is enough to have $\sum_{k\ge 1} w_k\le 1$ in the last move.
The $i^{\rm th}$ move $(v_{i_1},w_1),(v_{i_2},w_2), \dots $ is \emph{legal} if it satisfies the following condition:
\tmz 
\item[$(*)$]
For every $k \ge 1 $  there exists an edge $E\in\cE$ such that
 $v_{i_k} \in E$
and
$$\ell_{i-1}(E) + \Biggl(  \sum_{\substack{v_{i_s}\in E\\
 1\le s \le k-1}} w_s \Biggr) +w_k\le 1 \, .$$
\etmz
That is, for each legal submove there exists an edge $E$ such that its
load increases by exactly the weight in the submove.
This rule forces that no part of $w_k$ be wasted during the submove.

The cover function $t_{i}$ can gradually be reached from $t_{i-1}$
 by adding the weight $w_k$ to $t(v_{i_k})$ after each submove; this process
 converts also the corresponding load function from $\ell_{i-1}$ to $\ell_{i}$.

The \emph{value\/   $|\cG|$ of the game} is defined as the value $|t_q|$ of the cover function obtained at the end, that is the sum of the weights that have been spent during the game.
The goal of Edge-hitter is to achieve a value $|\cG|$ as small as possible, while Staller wants a large $|\cG|$.

Assuming that Edge-hitter starts the fractional transversal game on $\cH$, we consider the set of upper bounds,
 $$
   U_{\cH}=\{a: \mbox{Edge-hitter has a strategy that ensures } |\cG| \le a \}
 $$
 and the set of lower bounds,
 $$
   L_{\cH}=\{b: \mbox{Staller has a strategy that ensures } |\cG| \ge b \}.
 $$
 

Formally the \emph{game fractional transversal number} $\tau^*_g(\cH)$ is defined as
$$
\tau_g^*(\cH)=\inf (U_\cH) .
$$
The Staller-start game fractional  transversal number ${\tau_{g}^*}'(\cH)$
is defined similarly, under the condition that the first move is made by Staller.

The following assertion shows that $\tau^*_g(\cH)$ is in fact the common optimum for Edge-hitter and Staller, and so is ${\tau_{g}^*}'(\cH)$ as well. The proof is essentially the same as the one for the  game fractional domination number in \cite{BT-fr}.

\begin{prop}
	For every hypergraph\/ $\cH$ we have\/ $\inf (U_{\cH}) = \sup (L_{\cH})$, and the analogous equality holds for the Staller-start game, too.
\end{prop}
\Proof
First, assume that $\inf (U_{\cH}) < \sup (L_{\cH})$ and consequently, there exist two reals $x$ and $y$ satisfying $\inf (U_{\cH}) < x < y <\sup (L_{\cH})$. By definition, $x \in U_{\cH}$ and, therefore, Edge-hitter can ensure that, under every strategy of Staller, the value of the game is at most $x$.  Similarly, $y \in L_{\cH}$ and Staller has a strategy that ensures  $|\cG|\ge y$ whatever strategy is followed by Edge-hitter.
This is a contradiction that establishes  $ \inf (U_{\cH}) \ge \sup (L_{\cH})$.

Now, we prove the reverse inequality. By definition,  $z < \inf (U_{\cH})$ implies that Edge-hitter does not have a strategy to achieve $|\cG| \le z$.
That is, against each strategy of Edge-hitter there is a strategy of Staller which results in $|\cG| > z$. We may infer that $z\in L_{\cH}$ and therefore  $z\le \sup (L_{\cH})$. Since it holds for every $z< \inf (U_{\cH})$, we conclude $ \inf (U_{\cH}) \le \sup (L_{\cH})$.
This completes the proof of the proposition.
\qed

\bsk

Later, in Section~\ref{sec:algo}, we will show that $\inf (U_{\cH}) =  \min (U_{\cH})$ and $\sup (L_{\cH}) =  \max (L_{\cH})$. Therefore,  Edge-hitter and Staller have optimal strategies under which, respectively, $|\cG| \le \tau^*_g(\cH)$ and $|\cG| \ge \tau^*_g(\cH)$ are achieved.


\paragraph{Examples for the fractional transversal game.}
Our first example is the $4$-cycle $C_4:v_1v_2v_3v_4v_1$ which can also be considered as a $2$-uniform hypergraph.
Remark that $\tau^*(C_4)=2$ and we have $\tau_g(C_4)=3$ and $\tau'_g(C_4)=2$ for the integer games where Edge-hitter and Staller starts, respectively.
It can be proved that $\tau_g^*(C_4)=5/2$ and Edge-hitter has a strategy to ensure that the sum of the weights spent during the game is at most $5/2$.
One possible game under optimal strategies is the following.
The first move of Edge-hitter is $(v_1,\frac{1}{4}), (v_2,\frac{1}{4}), (v_3,\frac{1}{4}), (v_4,\frac{1}{4})$ that results $\ell_1 \equiv \frac{1}{2}$.
Then, Staller replies by playing $(v_2,\frac{1}{2}), (v_3,\frac{1}{2})$ that gives $\ell_2(v_iv_{i+1})=1$ for $i = 1,2,3$ while $\ell_2(v_4v_{1})$ remains $\frac{1}{2}$.
Then, Edge-hitter has to spend a weight $\frac{1}{2}$, for example by playing $(v_1,\frac{1}{2})$, to finish the game.

If Staller starts the fractional transversal game on $C_4$ with the move $(v_1,w_1), (v_2,w_2),\\ (v_3,w_3), (v_4,w_4)$ (or with any permutation of these submoves), then Edge-hitter can always ensure $|\cG|=2$ by playing  $(v_1,w_3), (v_2,w_4), (v_3,w_1), (v_4,w_2)$. Indeed, $\ell_2$ assigns $w_1+w_2+w_3+w_4=1$ to every edge of the graph. Therefore, ${\tau_{g}^*}'(C_4) \le 2$. Since ${\tau_{g}^*}'(C_4) \ge \tau^*(C_4)=2$ also holds, we get ${\tau_{g}^*}'(C_4)=2$. 

Now we modify the previous example $C_4$ by adding four new vertices $u_1, \dots, u_4$ and four new edges $\{v_1,v_{2}, u_1\}, \dots, \{v_4,v_1, u_4\}$ to get the hypergraph $\cH$. When the (integer or) fractional transversal number is considered each edge that is a superset of another edge can be deleted. This gives $\tau^*(\cH)=\tau^*(C_4)=2$. We show that the situation is different for the fractional transversal game on $\cH$.  Again, an optimal first move for Edge-hitter is  $(v_1,\frac{1}{4}), (v_2,\frac{1}{4}), (v_3,\frac{1}{4}), (v_4,\frac{1}{4})$, but here Staller may respond by playing $(u_1,\frac{1}{2}), (u_3,\frac{1}{2})$ that leaves $\ell_2(\{v_2,v_3,u_2\})= \ell_2(\{v_4,v_1,u_4\})=\frac{1}{2}$ and forces Edge-hitter to spend a weight of $1$ in his next move. It can be proved that the value $|\cG|=3$ equals $\tau^*_g(\cH)$ and therefore, $\tau^*_g(\cH)\neq \tau^*_g(C_4)$.

\section{Some basic facts and the Continuation Principle}
\label{sec:contprinc}

In this section we first observe some simple inequalities which are analogous to
 the ones in other games concerning graph domination and hypergraph transversal,
 most notably to the fractional domination game \cite{BT-fr}.

\begin{prop}  \label{prop:bounds} ~~~
	\begin{itemize}
        \item[$(i)$] For every hypergraph $\cH$, it holds that
        $$\tau^* (\cH) \le \tau_g^*(\cH) < 2 \tau^* (\cH) \qquad \mbox{and}  \qquad 
        \tau^* (\cH) \le {\tau_{g}^*}'(\cH) < 2 \tau^* (\cH)+1.$$
        \item[$(ii)$] There is no universal constant\/ $C$ with
		  $\tau_g(\cH) \le C\cdot \tau^*_g(\cH)$, and not even with\/
		   $\tau(\cH) \le C\cdot \tau^*_g(\cH)$.
		 The same holds true for\/ ${\tau_{g}^*}'(\cH)$, too.
			\end{itemize}
\end{prop}

\Proof
No matter which player starts the game, at the end the cover function
 $t_q$ is a fractional transversal.
This implies the lower bounds $\tau_g^*(\cH) \ge \tau^* (\cH)$ and
 ${\tau_{g}^*}'(\cH) \ge \tau^* (\cH)$.

Concerning a fractional transversal game $\cG$ on $\cH$ and the upper bounds in $(i)$,  we can write the value of the
 game in the form $|\cG|=W+W'$, where $W$ and $W'$ denote the total sum of weights
 assigned by Edge-hitter and Staller, respectively.
To keep the claimed bounds, first Edge-hitter can fix an optimal
 fractional transversal $f$, i.e.\ one with $|f|= \tau^*(\cH)$.
After that, in his moves he can apply the strategy to play submoves $(v_{i_j}, w_j)$
 with the largest possible weights $w_j$ which are not only allowed by $(*)$
  but also respect the inequalities $t_{i-1}(v_{i_j})+w_j \le f(v_{i_j})$.
If such a legal submove with a positive weight does not exist anymore, then $\cH$ is fully covered and the game is finished.

This strategy yields $W\le \tau^*(\cH)$, with strict inequality if
 the game is finished by Staller.
We also have $W'\le W$ or $W'\le W+1$, depending on whether the first
 move is made by Edge-hitter or Staller, both with strict inequalities if
 the game is finished by Edge-hitter.
Since only one of the players can make the last move, the claimed strict upper
 bounds follow.

For the proof of $(ii)$ we apply the following result of Alon \cite{A-1990}:
 For every fixed $\epsilon>0$ there is an integer $k$ and a $k$-uniform
  hypergraph $\cH=(V,\cE)$  such that  $\tau(\cH) \geq (1-\epsilon)\frac{\ln k}{k} (|V|+|\cE|)$.
On the other hand, a very simple fractional transversal $f$ with $|f|=|V|/k$
 may be constructed by assigning $f(v)=1/k$ to each vertex $v\in V$.
Therefore, $\tau^*(\cH) \le \frac{|V|}{k}$ and we obtain
  $$
    (1/2-\epsilon) \ln k < \sup_{\cH} \frac{\tau(\cH)}{2\tau^*(\cH)}
     < \sup_{\cH} \frac{\tau(\cH)}{\tau_g^*(\cH)}
     \leq \sup_{\cH} \frac{\tau_g(\cH)}{\tau_g^*(\cH)}
  $$
  due to the obvious fact $\tau \leq \tau_g$ and the inequality
   $\tau^*_g < 2\tau^*$ from $(i)$.
For ${\tau_{g}^*}'(\cH)$ the proof is similar, by the second part of $(i)$.
 \qed

\medskip  
 
\begin{prop}
 The upper bounds in Proposition~\ref{prop:bounds} $(i)$ are tight
  apart from an additive constant at most 2.
\end{prop}

\Proof
Consider the complete bipartite graph $G=K_{k,k^2}$ on $k+k^2$ vertices
 as a \hbox{$2$-uniform} hypergraph.
Clearly, $\tau^*(G)=k$. 
In any submove of a fractional transversal game, while $G$ is not fully covered, Staller can always select a vertex from the bigger partite class.
Following this strategy, during $k-1$ moves, Staller increases the sum of the loads by at most $(k-1)k$.
As $\Delta(G)=k^2$, $k-1$ moves of Edge-hitter increase the loads
 by at most $(k-1)k^2$.
Hence, no matter whether Edge-hitter or Staller starts the game, after $2k-2$ moves we have
 $$
   \sum_{E \in \cE} \ell_{2k-2}(E) \le (k-1)k+(k-1)k^2=k^3-k < |E| \,,
 $$
therefore the game is not over yet.
This shows $\tau_g^*(G) >2k-2= 2\tau^*(G) -2$ and, similarly,
 ${\tau_{g}^*}'(\cH)\ge 2\tau^*(G) -1$ follows if Staller starts the game.
 \qed

\medskip  

A monotone property of the game fractional transversal number is expressed
in the following idea, which provides a useful tool in simplifying
several arguments.
Let a hypergraph $\cH$ with a pre-defined load function $\ell$ be given,
 which we consider as a non-zero starting configuration.
We ask about the value $|\cG|$ of the game started by Edge-hitter,
 where the game is finished when $\ell$ is completed to a load function
 under which $\cH$ is fully covered.
The rules are the same as they were in the case of $\ell_0\equiv 0$,
 but here we have $\ell_0=\ell$, while we still start with $t_0\equiv 0$.
Assuming that both players play optimally, the value of the game will be
 denoted by $\tau_g^*(\cH| \ell)$ and termed the \emph{game\/ fractional\/
 $\ell$-transversal number}.
The corresponding Staller-start value ${\tau_{g}^*}'(\cH| \ell)$ is defined analogously.

\begin{thm}   \label{t:contpr}
	If\/ $\ell$ and\/ $\ell'$ are load functions on\/ the hypergraph $\cH=(V,\cE)$
	 such that\/ $\ell(E) \le \ell'(E)$ holds for every\/ $E\in \cE$, then\/
	$\tau_g^*(\cH| \ell)\ge \tau_g^*(\cH| \ell')$, and similarly\/
	${\tau_{g}^*}'(\cH| \ell)\ge {\tau_{g}^*}'(\cH| \ell')$.
\end{thm}

\Proof
 Assume for a contradiction that
 $\tau_g^*(\cH| \ell)<\tau_g^*(\cH| \ell')$,
     and choose a real $x$ with
      $\tau^*(\cH| \ell) < x < \tau_g^*(\cH| \ell')$.
     We use the imagination strategy between the following two games:

\msk

     {\sl Game 1}: Edge-hitter plays on $\cH|\ell$ applying a strategy which
     ensures that the length is at most $t_1 < x$.

\msk

    {\sl Game 2}: Staller plays on $\cH|\ell'$ applying a strategy which
     ensures that the length is at least $t_2 > x$.

\msk

Once this can be done, the assertion follows by the contradiction
  $x> t_1 > t_2 > x$.

     The moves essentially are copied between Game 1 and Game 2 such that
     $\ell(E)\le \ell'(E)$ remains true for all $E\in \cE$ after every move.
If this inequality is valid before Staller's move, then her next move in $\cH|\ell'$
 is feasible in $\cH|\ell$ as well, so that Edge-hitter can simply copy it
 and reply to it.
On the other hand it may happen that one or more submoves $(v_{i_k},w_{k})$
 made by Edge-hitter in Game 1 are not feasible in Game~2.
Observe, however, that if this happens, then already
 all loads on the edges incident with $v_{i_k}$ reach 1 in $\cH|\ell'$.
It follows that those loads will never become smaller than the corresponding
 ones in $\cH|\ell$.
Consequently, $\tau_g^*(\cH| \ell)\ge \tau_g^*(\cH| \ell')$ holds.

The analogous conclusion can be reached in the Staller-start game as well,
 literally by the same argument, deriving a contradiction from the assumption
 ${\tau_{g}^*}'(\cH| \ell) < {\tau_{g}^*}'(\cH| \ell')$.
\qed

\bsk

We obtain the following immediate consequence.

\begin{thm}
	The  game fractional transversal numbers for the Staller-start
	and for the Edge-hitter-start games on\/ $\cH$ may differ by at most\/ $1$.
\end{thm}
\Proof
 Consider the Staller-start game.
Whatever Staller moves first, she assigns total weight 1, and
 creates a situation which is at least as
 favorable for Edge-hitter as the all-zero load at the beginning of the
 original transversal game.
Then, due to Theorem \ref{t:contpr}, Edge-hitter can ensure that the game ends
 using at most $\tau^*_g(\cH)$ further weight.
This proves ${\tau_{g}^*}'(\cH) \le \tau_g^*(\cH) + 1$.

Similarly, if Edge-hitter starts, after his first move he is in
 at least as favorable position as with the all-zero load at the beginning
 of the Staller-start game.
This proves the reverse inequality $\tau_g^*(\cH) \le {\tau_{g}^*}'(\cH) + 1$.
\qed

\section{Infinite moves are not necessary} \label{sec:infinite}


The definition of a legal move in the transversal game admits the option that a player splits the value 1 into an infinite number of pieces; e.g., $w_k=2^{-k}$. 
It turns out, however, that each legal move on $\cH=(V,\cE)$ is equivalent to
 a move which consists of at most $|V|$ submoves.

\begin{thm} \label{finite}
	Every legal move in a fractional transversal game can be replaced with a legal move such that each vertex occurs in at most one submove of it and the two moves result in the same load function. 
\end{thm}

\Proof
First, consider a vertex $v$ which occurs in two different submoves $(v_{i_j}, w_{j})$
 and $(v_{i_k}, w_{k})$ of a move.
That is, $v=v_{i_j}=v_{i_k}$ and we may assume $j < k$.
By the condition $(*)$, there exists an edge $E\in \cE$ such that $v \in E$ and
 the second submove $(v_{i_k}, w_{k})$ increases the load of $E$ by exactly $w_{k}$.
If the submove $(v_{i_j}, w_{j})$ is deleted from the sequence and the weight
 $w_{k}$ is replaced by $w_{j}+w_{k}$ in the $k^{\mathrm{th}}$ submove,
  the submove and the whole move remain  legal and result 
   in the same load function as before.
Performing this modification repeatedly we can achieve that every vertex occurs
 in either zero or exactly one or infinitely many submoves of the move in question.
This already proves the statement if the move contains only a finite number of submoves.

Now, assume that the move is infinite.
Then, the sequence of submoves can be split into two, such that the first part
 is finite, and in the second infinite part every vertex (which is present there)
 is repeated infinitely many times.
Consider this infinite subsequence $S=(v_{i_s},w_s), \dots $.
By renaming the vertices of $\cH$ if necessary, we may assume that
$
\{v_1, \dots ,v_\ell\}$ is the set of the vertices which are present in $S$.
We prove that the finite sequence $S'=(v_1, \sum_{j:\,\,i_j=1}w_j),\dots , (v_\ell, \sum_{j:\,\,i_j=\ell}w_j)$ is equivalent to $S$.
Clearly, $S$ and $S'$ yield the same load function after the move.
So, it is enough to prove that $S'$ is legal.
Assume for a contradiction that $(v_k, \sum_{j:\,\,i_j=k}w_j)$ is not a legal submove
 in $S'$, and let $k$ be the smallest such index.
Then, after the $(k-1)^{\rm st}$ submove of $S'$, every edge $E$ which contains $v_k$ has
 a load $\ell(E) > 1- \sum_{j:\,\,i_j=k}w_j$ and, moreover, there is a positive constant
  $\epsilon$ such that $\min_{v_k\in E} \ell(E) + \sum_{j:\,\,i_j=k}w_j = 1+\epsilon.$
Now, consider $S$ again.
There is an index $p=p(\epsilon)$ such that   $\sum_{j\ge p}w_j <\epsilon$ and hence, before the $p^{\mathrm{th}}$ submove of $S$, each edge containing $v_k$ is fully covered.
As $v_k$ occurs infinitely often in $S$, and also the occurrences after the $p^{\mathrm{th}}$ submove are legal, this is a contradiction. \qed
\medskip

We say that a legal finite move $(v_{i_1}, w_{1}), \dots, (v_{i_k}, w_{k})$ is \textit{transposable}
 if for any permutation $\beta(1), \dots , \beta(k)$ of $1,\dots ,k$, the move $(v_{i_{\beta(1)}}, w_{\beta(1)}), \dots, (v_{i_\beta(k)}, w_{\beta(k)})$ is also legal.  

\begin{prop}
A legal move\/ $(v_{i_1}, w_{1}), \dots, (v_{i_k}, w_{k})$, where\/
 $w_j>0$ for all\/ $1\leq j\leq k$ and no vertices are repeated,
  is transposable if and only if every vertex\/ $v_{i_j}$\/ is incident with
  an edge\/ $E_{i_j}$ such that\/ $\sum_{v_{i}\in E_{i_j}} t(v_i)\leq 1$
   holds after performing the entire move.
\end{prop}

\Proof
If the inequality holds for $E_{i_j}$, then omitting the submove $(v_{i_j}, w_{j})$
 from the move we obtain
  $\sum_{v_i\in E_{i_j} \setminus \{v_{i_j}\}} t(v_i)\leq 1-w_{j}$.
Hence $(v_{i_j}, w_{j})$ is a legal submove no matter when it is performed
 during the move.
This means that the move is transposable
 whenever the condition is satisfied for all $j$.

In the other direction assume that for some $v_{i_j}$ every edge violates
 the inequality.
Violation means that the corresponding sum is at least $1+\epsilon$ on every edge,
 for a certain fixed $\epsilon>0$.
Consider now a permutation in which $(v_{i_j}, w_{j})$ is the last submove.
Then $(v_{i_j}, w_{j})$ is not legal because at most $w_j-\epsilon$ can be
 assigned to $v_{i_j}$ legally.
Consequently the move is not transposable.
\qed

\begin{thm}
	\label{prop:transp}
  If a legal move\/ $m=(v_{i_1}, w_{1}), \dots, (v_{i_k}, w_{k})$ is not transposable
   in the fractional transversal game, then it can be replaced by a legal
    transposable move after which no edge gets smaller load than after\/ $m$. 
\end{thm}
\Proof
First, consider the legal move $m=(v_{i_1}, w_{1}), \dots, (v_{i_k}, w_{k})$ and the move $m'$  which is obtained by the permutation $\beta=2, \dots, k, 1$.
It is clear by definition that the first $k-1$ submoves of $m'$ remain legal. For the last (and  not necessarily legal) submove, determine $w_1^*$ as the maximum weight which results in a legal last submove with $v_{i_1}$.
We have $0 \le w_1^* \le  w_1$. If $w_1=w_1^*$, then $m'$ is legal and gives exactly the same load function as $m$. If $w_1 > w_1^*$, then the same load function is obtained after the submove $(v_{i_1}, w_1^*) $ as after $m$, because in both cases every edge incident with $v_{i_1}$ is fully covered
 and the loads of the other edges are unchanged.
In this latter case, the sum of the weights used is  $1-(w_1-w_1^*) <1$.
After this change,  the submove $(v_{i_1}, w_1^*)$ will be legal in any permutation of $(v_{i_1}, w_{1}), \dots, (v_{i_k}, w_{k})$.
That is, if a permutation is not feasible after this replacement, then infeasibility
 is due to another vertex.
The same is true if some weights $w_s$ are replaced by smaller weights.

We repeat this
 step for the modified sequence with permutation $\beta=3, \dots, k, 1, 2$,
 then with $\beta=4, \dots, k, 1, 2, 3$, and so on, finally with
 $\beta=k, 1, 2, \dots, k-1$, keeping all modifications incrementally.
Decreasing the weight of the last submove in each step if necessary, at the end a
 legal transposable move $m^*$ is obtained, which yields the same load function as
  $m$ and satisfies $\sum_{j=1}^k w_j^* \le \sum_{j=1}^k w_j$.
If the game is not over yet and if $s^*= \sum_{j=1}^k w_j-\sum_{j=1}^k w_j^*$ is positive, we may use the weight $s^*$ to increase the loads of the non-fully covered edges.
This may yield a move which is not transposable.
If so, then we repeatedly perform the steps of modifiation as described above.

If the process terminates after a finite number of iterations, then the last
 version of the move is transposable, due to the way as the steps have been defined.
Otherwise we obtain an infinite sequence $s_1^*,s_2^*,\dots$ of re-distributions
 from the total unit weight of the move.
Note that the total load of edges increases by at least $\sum_{i\geq 1} s_i^*$,
 hence this sum is convergent because altogether the total load is at most $|\cE|$.
On the other hand, the weight of each vertex changes by at most $s_i^*$ in the
 $i^{\mathrm{th}}$ iteration, hence the local changes in weight (one of which is negative
 in each iteration) sum up to at most $\sum_{i\geq 1} s_i^*$, therefore they
  are also convergent.
It means that for every $\epsilon>0$ there exists a threshold $i(\epsilon)$
 such that $\sum_{i\geq i(\epsilon)} s_i^* < \epsilon/k$.
That is, the cummulative change of weight at each vertex obtained later than
 iteration $i(\epsilon)$ is smaller than $\epsilon/k$, and the change in
  the sum of weights inside each edge is smaller than $\epsilon$.

Let $w^*$ denote the limit sequence of the $k$ weights of the $k$ submoves
 under this infinite process,
 i.e.\ the submove at $v_j$ tends to weight $w_j^*$ for $j=1,\dots,k$
 (allowing that some of the $w_j^*$ are zero).
We claim that the move $m^*=(v_1,w_1^*),\dots,(v_k,w_k^*)$
 is legal and transposable.
This will complete the proof because the loads never decrease, hence
 under $w^*$ no edge gets smaller load than by move $m$.
Let us also denote by $t^*$ the cover function obtained after
 $m^*$, disregarding for the moment that the move is not yet
 proved to be legal.

Suppose for a contradiction that some permutation $\beta$ yields a move
 which is not legal.
We may assume without loss of generality that $\beta=1,2,\dots,k$ and that the
 submove $(v_{k}, w_{k}^*)$ is not legal.
It means $\sum_{v_j\in E} t^*(v_j)\geq1+\epsilon$ for all edges $E$
 containing $v_k$, for some constant $\epsilon>0$.
We know, however, that if we subtract the distributed weight $s^*_{i(\epsilon)}$
 from the vertex loads in iteration $i(\epsilon)$, we obtain a legal move;
 and compared to that situation, each vertex weight changed by
 no more than $\sum_{i\geq i(\epsilon)} s_i^* < \epsilon/k$.
Consequently there exists an edge $E\in\cE$ such that $v_k\in E$ and
 $$
   \sum_{v_j\in E} \biggl( t^*(v_j) - \sum_{i\geq i(\epsilon)} s_i^* \biggr)
    \leq 1 ,
 $$
 $$
   \sum_{v_j\in E}  t^*(v_j) 
    < |E|\cdot\epsilon/k < 1+\epsilon .
 $$
This contradiction completes the proof of the theorem.
\qed

\begin{rmk}\label{remark:transposable}
Based on Theorem~\ref{prop:transp}, Edge-hitter may restrict his strategy to transposable moves. On the other hand, the result suggests that Staller is advised to perform moves, if possible, which are `very non-transposable' in a sense.
\end{rmk}

\section{Algorithm for computing the value of the game}
\label{sec:algo}

We consider an equivalent version of the game,
the \emph{structured game},
which is easier to
analyze.

Each \emph{move} consists of $n+1$ \emph{rounds}.
Each round consists of $n$ \emph{submoves}, which are dedicated to the vertices
 $v_1,\dots,v_n$ in succession.
In each submove, the player whose turn it is can decide the amount $w$,
 the weight spent in the submove,
by which the cover value of $t(v_i)$ is increased, subject to the usual rules:
The increase must be useful, i.e.\ each submove must satisfy the condition
 $(*)$, and it must be within the \emph{budget
constraint} of total weight 1 to be spent per move.
It is possible to skip a submove by simply choosing $w$ to be zero.

The first $n$ rounds are identical, but the last round is special: In
each submove, the weight is greedily chosen as the largest possible
 legal weight, hence not allowing any freedom in choosing $w$ for the
  player in those submoves.
This ensures that the whole move spends a
total weight of 1 unless a cover is obtained.

There are $n$ moves, alternating between the two players.
This is enough to ensure that a cover is constructed when the game terminates.
Every move consists of $n^2+n$ submoves, and in total, the game consists of $N=n^3+n^2$ submoves.
We illustrate this for a small example with $n=4$ vertices, where the
Edge-hitter starts.
The sequence of submoves is
\begin{tabbing}
  \quad\=\+
$
  H_1,H_2,H_3,H_4;$ \ \=$
  H_1,H_2,H_3,H_4;$ \ \=$
  H_1,H_2,H_3,H_4;$ \ \=$
  H_1,H_2,H_3,H_4;$ \ \=$
  G_1,G_2,G_3,G_4;$\\$
  S_1,S_2,S_3,S_4;$\>$
  S_1,S_2,S_3,S_4;$\>$
  S_1,S_2,S_3,S_4;$\>$
  S_1,S_2,S_3,S_4;$\>$
  G_1,G_2,G_3,G_4;$\\$
  H_1,H_2,H_3,H_4;$\>$
  H_1,H_2,H_3,H_4;$\>$
  H_1,H_2,H_3,H_4;$\>$
  H_1,H_2,H_3,H_4;$\>$
  G_1,G_2,G_3,G_4;$\\$
  S_1,S_2,S_3,S_4;$\>$
  S_1,S_2,S_3,S_4;$\>$
  S_1,S_2,S_3,S_4;$\>$
  S_1,S_2,S_3,S_4;$\>$
  G_1,G_2,G_3,G_4.$
\end{tabbing}
Here $H_i$ denotes a move of Edge-hitter for vertex $i$, and
$S_i$ denotes a move of Staller for vertex~$i$.
The greedy moves are denoted by $G_i$.

We do not stipulate as part of the rules that the whole budget of 1
unit must be spent during a move. This capacity is only an upper
bound.  It is still true that the whole budget is spent in each move
if the game is played from the beginning.  However, this arises as a
\emph{consequence} of the new setup, due to the greedy moves.

The \emph{duration} of a play is defined as the total weight
spent during all submoves.
As soon as a cover is found, the rules imply that no more weight can
be spent, and thus the game is effectively over.

\begin{lem}
  The structured game has the same value as the original game.
\end{lem}
\begin{proof}
  According to Theorem~\ref{finite}, we can assume that
  every vertex occurs at most once in a move.
  We can realize this in the structured game by
  selecting one vertex per round and leaving the weight at 0
  otherwise.
  Thus, the structured game does not restrict the players' strategies,
  when compared to the original game.
  On the other hand, the structured game does not give the players
  more power:
  The greedy moves ensure that the total weight of 1 is used as long
  as it is possible.
\end{proof}
In Remark~\ref{remark:transposable} we have argued that, for the
Edge-hitter, any permutation will do.
Thus, a single round
$H_1,H_2,\ldots,H_n$ followed by $n$ greedy submoves would be sufficient for
Edge-hitter.  For simplicity, we have however chosen to treat the two players uniformly.
Note that the greedy submoves are necessary also in case of Edge-hitter.
Otherwise, for example, he might pass the first move and transform the game to the Staller-start version which sometimes admits a smaller game value
 	(cf.\ the example of $C_4$).

Consider the situation after the $j^{\mathrm{th}}$ submove, $0\le j \le N$.
Let $\vec x\in [0,1]^{\cE}$ 
 be an arbitrary load vector, and let $r\in [0,1]$
be the budget for the current move that is still available.
If $j=k(n^2+n)+i$, $r$ is the total weight spent
in the last $i$~submoves.

We define
\begin{displaymath}
  T_j(\vec x,r)
\end{displaymath}
as the remaining duration of the game
if both players play
optimally, starting from the current situation.
If $j$ is not very small then for a certain region of $\vec x$
 it happens that a complete fractional cover
 is not reached because the game necessarily ends after the last submove
 of the $n^{\mathrm{th}}$ round.\footnote{Remark that
  if the $j^{\mathrm{th}}$ submove is the last submove in a move of Staller
   and $\ell_x$ is the load function corresponding to $\vec x$,
 then we would expect $T_j(\vec x,1)$ to be
$\tau_g^*(\cH| \ell_x)$, but this cannot be true in general because
$\tau_g^*(\cH| \ell_x)$ may be larger than the number of remaining moves.
}
This convention
 has the advantage that $T_j$ is
defined for arbitrary $\vec x$ and $r$.

The value of the whole game is $T_0(\vec 0,1)$.

We will derive a backward recursion for the functions $T_j$, and thus
show that they are piecewise linear and continuous.

Given $j$, we know the type of the $j^{\mathrm{th}}$ submove ($H$, $S$, or $G$) and the
vertex $v_i$ to which it applies.
We denote the maximum permitted weight by 
\begin{equation}
  \label{eq:wmax}
w^{\max}_i(\vec x,r) = \min \{r, \max\{\,1-x_E\mid E\ni v_i\,\}\} \,.
\end{equation}
For the result of increasing the cover value 
 of $v_i$ by $w$ we write
\begin{equation}
  \label{eq:update}
  \textit{update}_i(\vec x,w) = \vec x\,'
\text{ \ with \ }
x'_E =
\begin{cases}
  x_E, & \text{if $v_i\notin E$,}\\
  \min\{1,x_E+w\}, & \text{if $v_i\in E$.}\\
\end{cases}
\end{equation}
With these definitions,
the recursion for a submove $H_i$ for Edge-hitter can be written easily:
\begin{equation}
  \label{eq:Edge-hitter}
  T_{j-1}(\vec x,r) = 
    \min\{\, w+ T_{j}(\textit{update}_i(\vec x,w),r-w) \mid 
    0 \le w \le w^{\max}_i(\vec x,r)\,\}
\end{equation}
If the submove is for the Staller ($S_i$), the recursion is the same as
\eqref{eq:Edge-hitter}, except that $\min$ is replaced by
 $\max$.
In the greedy submoves $G_i$, we always choose
$w = w^{\max}_i(\vec x,r)$:
\begin{equation}
  \label{eq:greedy}
  T_{j-1}(\vec x,r) = 
w^{\max}_i(\vec x,r)+ T_{j}(\textit{update}_i(\vec x,w^{\max}_i(\vec x,r)),r-w^{\max}_i(\vec x,r)
)
\end{equation}
The last greedy submove $G_n$ of each move is an exception: Since a
different move is about to start, the budget $r$ is reset to 1.
Thus, when $j$ is a multiple of $n^2+n$, then
\begin{equation}
  \label{eq:greedy-n}
  T_{j-1}(\vec x,r) = 
w^{\max}_n(\vec x,r)+
 T_{j}(\textit{update}_n(\vec x,w^{\max}_n(\vec x,r)),0).
\end{equation}
As the recursion anchor, we use the
value after the final move, which is simply
\begin{equation}
  \label{eq:final-T}
  T_N(\vec x,r)=1.
\end{equation}


\begin{thm}\label{rational}
  Each function $T_j(\vec x,r)$ for $0\le j \le N$ is a piecewise
  linear continuous
function with finitely many linear pieces defined on 
$[0,1]^{\cE}\times[0,1]$.
Moreover, all\/ $T_j$ are rational in the sense that each linear piece has
rational coefficients and rational constant part.
As a consequence, the boundaries between regions of the domain
 with different linear functions
can be described by rational equations.
\end{thm}

\begin{proof}
We will call a function with all the desired properties
---
piecewise linear, continuous, and
 rational, with finitely many linear pieces
---
a
PLCR function.

The proof proceeds by backward recursion from $T_N$ down to $T_0$.
The function $T_N$ from \eqref{eq:final-T} is obviously PLCR.

The sum, difference, maximum, or minimum of two PLCR functions
is again PLCR,
and the same holds true when
substituting one PLCR function into another.
It follows directly that the functions $w_i^{\max}$ and $\textit{update}_i$ are
PLCR functions on the domain $[0,1]^{\cE}\times[0,1]$.
This allows us to perform the induction step in the recursions
 (\ref{eq:greedy})--(\ref{eq:greedy-n}) for~$G_i$.

In the recursion
\eqref{eq:Edge-hitter} we additionally have a minimization (or, in the
analogous recursion for Staller, a maximization) over some range of values~$w$.
It has the form
\begin{displaymath}
  \min \{\,F(\vec x,r,w) \mid 0\le w \le w^{\max}_i(\vec x,r)\,\}
\end{displaymath}
with the 
PLCR function
 \begin{equation*}
F(\vec x,r,w) := w+
T_{j}(\textit{update}_i(\vec x,w),r-w)
\end{equation*}
To get rid of the varying upper bound on $w$, we rewrite the recursion in terms of
another PLCR function
\begin{equation*}
\hat F(\vec x,r,w) =
F(\vec x,r,\min\{w,w^{\max}_i(\vec x,r)\}
\end{equation*}
as
\begin{equation*}
  T_{j-1}(\vec x,r) = \min\{\,\hat F(\vec x,r,w) \mid 0\le w\le 1\,\}.
\end{equation*}
Lemma~\ref{PLCR} below establishes that $T_{j-1}$ is
 a PLCR function.

The same argument applies to the recursion
 for the Staller ($S_i$),
where $\min$ is replaced by $\max$.
\end{proof}

\begin{lem} \label{PLCR}
  Suppose that
$\hat F(y,w)\colon
[0,1]^m\times[0,1]\to \mathbb{R}
$ is a PLCR function.
Then the function
$ T(y)\colon
[0,1]^m\to \mathbb{R}
$ defined by minimizing over $w$:  
  \begin{equation}\label{T}
  T(y) := \min\{\,\hat F(y,w) \mid 0\le w\le 1\,\}
\end{equation}
 is also a PLCR function.
\end{lem}
 \begin{proof}
   We first show that $T$ is continuous.
   Since $\hat F$ is PLCR, it is Lipschitz-continuous.
Let $L$ be its Lipschitz constant with respect to the $\infty$-norm.
(We can compute $L$ as the maximum $L_1$-norm of all
coefficient vectors of the linear pieces of
$\hat F$.)
It follows that the function $T$ in \eqref{T}
is also Lipschitz-continuous with Lipschitz-constant~$L$. To see this,
let $\lVert y_0-y_1\rVert \le \eps$, and let
$T(y_0)=\hat F(y_0,w_0)$ for some~$w_0$.
Then
$T(y_1)\le \hat F(y_1,w_0)\le F(y_0,w_0) + L\eps=T(y_0) + L\eps$.
The converse bound
$T(y_0)\le T(y_1) + L\eps$ follows in the same way.

We still need to show that $T$ is piecewise linear.  For an intuitive
way to see this, one can interpret the minimization over $w$
geometrically. The graph of
$\hat F\colon [0,1]^m\times[0,1]\to \mathbb{R} $ is a subset of
$\mathbb{R}^{m+2}$.  Taking the minimum over all $w$ amounts to
projecting away the coordinate corresponding to~$w$ and taking the
lower envelope (with respect to the last coordinate) in the projection
in $\mathbb{R}^{m+1}$.  Figure~\ref{fig:lower-envelope} shows a
two-dimensional illustration.
This picture can also be interpreted as a
three-dimensional view of the graph of a bivariate function
$\hat F(y,w)$ when the viewing direction is parallel to the $w$-axis.
(In this
hypothetical example, the resulting minimum is
discontinuous; this cannot happen when $\hat F$ is continuous and its domain
is the box $[0,1]^m\times[0,1]$.)
\begin{figure}[htb]
  \centering
  \includegraphics{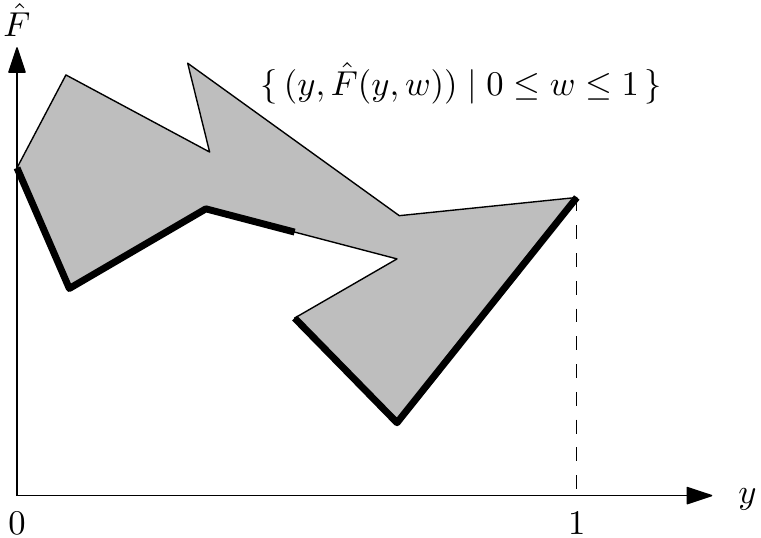}
  \caption{The lower envelope of a polyhedral set in 2 dimensions ($m=1$).}
  \label{fig:lower-envelope}
\end{figure}

Formally, we conduct the proof as follows.
 We know that the domain
$[0,1]^{m+1}$ 
of $\hat F$ splits into
 finitely many rational convex $(m+1)$-dimensional polytopes~$P$
on which $\hat F$ is linear:
\begin{displaymath}
  \hat F(y,w) = a_P y + b_P w + c_P,\quad \text{ for }(y,w)\in P
\end{displaymath}
for some rational coefficient vector $a_P$ and coefficients $b_P$ and $c_P$.
We can thus write $T(y)$ as the minimum of finitely many functions
$T_P(y)$ of the form
\begin{equation}
  \label{eq:T_P}
  T_P(y)
  := \min\{\, a_P y + b_P w + c_P
  \mid 0\le w\le 1, (y,w)\in P\,\},
\end{equation}
where the minimum of an empty set is taken as $\infty$.

For fixed $y$, the minimum
in \eqref{eq:T_P}
depends on the sign of $b_P$. If $b_P>0$, the minimum
is achieved on a boundary point that lies on some facet
$P'$ of $P$ whose outer normal has negative $w$-coordinate.
On such a facet, $w$ can be expressed as a linear function of~$y$, and
thus, $T_P$ can be written as a linear function
\begin{equation}
  \label{eq:T_P'}
  T_{P'}(y) = a_{P'} y + c_{P'},\quad \text{ for } y \in \bar P',
\end{equation}
where $\bar P'$ is the projection of the facet $P'$ to $[0,1]^m$.
Thus,
$T_P(y)$ is the minimum of finitely many functions
$T_{P'}(y)$, with the understanding that
$T_{P'}(y)$ is taken as $\infty$ when $y$ is outside its domain $\bar
P'$.

The situation is similar for $b_P<0$. When $b_P=0$, then $\hat F$ does
not depend on $w$ and we can simply write
\begin{equation}
 \label{eq:T_Pbar}
  T_{P}(y) = a_{P} y + c_{P},\quad \text{ for } y \in \bar P,
\end{equation}
where $\bar P$ is the projection of~$P$.

In summary, the function $T(y)$ can be written as the minimum of
finitely many pieces $T_P(y)$, each of which can in turn be written as
the minimum of finitely many \emph{linear} pieces
\eqref{eq:T_P'}
or~\eqref{eq:T_Pbar}.
All these pieces have rational coefficients and rational domain
boundaries, and since continuity of $T$ has already been established,
the PLCR property of $T$ follows.
 \end{proof}

 The proof of Theorem~\ref{rational} is constructive and, in
 principle, it provides an algorithm for computing 
 the value $T_0(\vec 0,1)$ of the game.
From this, we obtain the following important corollary.

\begin{thm}
 For every finite hypergraph\/ $\cH=(V,\cE)$, the game fractional transversal
  number\/ $\tau_g^*(\cH)$ and its Staller-start version\/ ${\tau_g^*}'(\cH)$
  are rational, and both players have optimal
 strategies in which they play only rational weights.
\end{thm}

\begin{rmk}
 It is not true in general that every optimal strategy uses
  only rational weights.
 A simple counterexample is the graph\/ $C_4$ where Staller can start with
  placing\/ $x$ and\/ $1-x$ on two vertices with any\/ $x\in[0,1]$,
  no matter if\/ $x$ is rational or irrational.
\end{rmk}

\section{Concluding remarks and open problems}
\label{sec:concl}

Putting the fractional domination game \cite{BT-fr} into a more general
 context, in this paper we introduced the fractional transversal game
 on hypergraphs.
Among other results we proved that the game value is rational, and both
 players have optimal strategies using rational weights and with a
 finite number of submoves.
Since a dominating set is a transversal of the closed neighborhood
 hypergraph, and total a dominating set is a transversal of the
 open neighborhood hypergraph, the following consequence is immediate.

\begin{thm}
 The fractional versions of both the domination game and the total
  domination game have rational game values on every hypergraph.
 Moreover, in either of these games, both players can achieve their
  common optimum using only rational weights and within finitely many
  submoves.
\end{thm}

We conclude this paper with some conjectures.

\begin{conj}   \label{conj1}
	If each of the first\/ $2k-1$ $($$k \ge 1$$)$ moves was an integer move in the fractional transversal game,
	i.e.\ of the form\/ $(v_{i_1}, 1)$, then Staller has an integer move
	in the\/ $(2k)^{\rm th}$ turn, which is optimal in the fractional transversal game.
\end{conj}

This means that fractional moves would be advantageous for Edge-hitter only.
If true then this conjecture implies the following weaker one.

\begin{conj}
	\label{conj2}
	For every hypergraph\/ $H$, \enskip
	$\tau^*_g(H) \le \tau_g(H).$
\end{conj}

Perhaps the following stronger version of Conjecture \ref{conj1} is also true.

\begin{conj}
	\label{conj3}
	Starting from any partial covering function,
	there is an optimal strategy for Staller
	where, in every submove, she
	always spends the largest
	legal weight.
\end{conj}


\paragraph{Acknowledgements.}

We are grateful to G\'abor Tardos for helpful discussions.
The first author acknowledges the financial support from the Slovenian
 Research Agency under the project N1-0108.
Research of the third author was supported in part by the National Research,
 Development and Innovation Office -- NKFIH under the grant SNN 129364.


\begin{thebibliography}{99}
	
	\bibitem{A-1990}
	N. Alon, Transversal numbers of uniform hypergraphs.
	{\it Graphs Combin.} 6 (1990) 1--4.
	
	\bibitem{borowiecki-2019}
	M.~Borowiecki, A.~Fiedorowicz, and E.~Sidorowicz,
	Connected domination game.
	{\it Appl.\ Anal.\ Discrete Math.}\  13 (2019) 261--289.
	
	\bibitem{BBGK-2019}
	B.~Bre\v{s}ar, Cs.~Bujt\' as, T. Gologranc, S.~Klav\v zar, G.~Ko\v{s}mrlj, T.~Marc, B.~Patk\'os, Zs.~Tuza, and M.~Vizer,
	 The variety of domination games. 
	 \textit{Aequationes Mathematicae} 93 (2019) 1085--1109.
	
	\bibitem{BDKK}
	B.~Bre\v{s}ar, P.~Dorbec, S.~Klav\v zar, and G.~Ko\v{s}mrlj,
	Domination game: effect of edge- and vertex-removal.
	\textit{Discrete Math.} 330 (2014) 1--10.
	
	\bibitem{BHKR-book}
	B.~Bre{\v{s}}ar, M.A.~Henning,  S.~Klav{\v{z}}ar, and D.F.~Rall,
	Domination Games Played on Graphs.
	SpringerBriefs in Mathematics, 2021.

		\bibitem{BKR-2010}
	B.~Bre{\v{s}}ar, S.~Klav{\v{z}}ar, and D.F.~Rall, Domination game
	and an imagination strategy. {\it SIAM J. Discrete Math.} 24 (2010)
	979--991.
	
	\bibitem{bujtas-2018}
	Cs.~Bujt\'as, 	General upper bound on the game domination number.
	\textit{Discrete Appl. Math.} 285 (2020) 530--538.
	
	\bibitem{BHT-2016}
	Cs.\ Bujt\'as, M.A. Henning, and Zs. Tuza,
	Transversal game on hypergraphs and the $\frac{3}{4}$-conjecture
	  on the total domination game.
	{\it SIAM J. Discrete Math.} 30 (2016) 1830--1847.
	
	\bibitem{BHT-2017}
	Cs.\ Bujt\'as, M.A. Henning, and Zs. Tuza,
	Bounds on the game transversal number in hypergraphs.
	{\it European J. Combin.} 59 (2017) 34--50.
	
	\bibitem{BIK-2021}
	Cs.~Bujt\'as, V.~Ir\v si\v c, and S.~Klav{\v{z}}ar,
	 Perfect graphs for domination games.
	{\it Ann. Comb.} 25 (2021) 133--152.

	\bibitem{BPTV}
	Cs.\ Bujt\'as, B. Patk\'os, Zs. Tuza, and M. Vizer,
	Domination game on uniform hypergraphs. 
	{\it Discrete Appl. Math.} 258 (2019) 65--75.
	
	\bibitem{BT-2016}
	Cs.~Bujt\'as and Zs.~Tuza, 
	The disjoint domination game.
	\textit{Discrete Math.} 339 (2016) 1985--1992.
	
	\bibitem{BT-fr}
	Cs.\ Bujt\'as and Zs. Tuza,
	Fractional domination game.
	{\it Electronic J. Combin.} 26 (2019) \#P4.3, 17 pp.
	
		\bibitem{henning-2016}
	M.A.~Henning and W.B.~Kinnersley,
	Domination game: a proof of the $3/5$-conjecture for graphs with minimum degree at least two.
	{\it SIAM J. Discrete Math.}\ 30 (2016) 20--35.
	
	\bibitem{HKR-2015}
	M.A. Henning, S. Klav\v{z}ar, and D.F. Rall,
	Total version of the domination game.
	{\it Graphs Combin.}  31 (2015)  1453--1462.
	
		\bibitem{HKR-2017} M.A. Henning, S. Klav\v{z}ar, and D.F. Rall,
	The 4/5 upper bound on the game total domination number.
	{\it Combinatorica\/} 31 (2017) 223--251.
	
	\bibitem{I-2019}
	 V.~Ir\v si\v c,
	Effect of predomination and vertex removal on the game total domination number of a graph. 
	\textit{Discrete Appl. Math}. 257 (2019) 216--225.
	
	 \bibitem{kinnersley-2013}
	 W.B.~Kinnersley, D.B.~West, and R.~Zamani,
	 Extremal problems for game domination number.
	 \textit{SIAM J. Discrete Math.} 27 (2013) 2090--2107.	
	 
	
	\bibitem{KR-2019}
	S.~Klav\v{z}ar and D.F.~Rall,
	Domination game and minimal edge cuts.
\textit{Discrete Math.}\ 342 (2019) 951--958.


	\bibitem{Kos-2017}
G. Ko\v smrlj, Domination game on paths and cycles. {\it Ars Math. Contemp.} 13 (2017)
125--136.

	\bibitem{XLK-2018}
K. Xu, X. Li, and S. Klav\v zar,
 On graphs with largest possible game domination number.
{\it Discrete Math.} 341 (2018) 1768--1777.
	
	

	
\end{thebibliography}
\end{document}